\documentclass[11pt,reqno,a4paper]{amsart}

\usepackage{amsmath,amsfonts,amsthm,mathrsfs,amssymb}
\usepackage[usenames]{color}

\usepackage{enumerate}
\usepackage{hyperref}
\usepackage{xcolor}
\hypersetup{
  colorlinks,
  linkcolor={blue!80!black},
  urlcolor={blue!80!black},
  citecolor={blue!80!black}
}
\usepackage{bm}

\numberwithin{equation}{section}
\usepackage[capitalize,noabbrev]{cleveref}

\usepackage{doi}
\usepackage[sort,numbers]{natbib}

\newcommand{\p}{\partial}

\theoremstyle{plain}
\newtheorem{theorem}{Theorem}[section]
\newtheorem*{theorem*}{Theorem}
\newtheorem{lemma}[theorem]{Lemma}

\newtheorem{proposition}[theorem]{Proposition}
\newtheorem*{proposition*}{Proposition}

\newtheorem*{conjecture*}{Conjecture}

\theoremstyle{definition}
\newtheorem{definition}[theorem]{Definition}

\theoremstyle{remark}
\newtheorem{remark}[theorem]{Remark}
\newtheorem*{remark*}{Remark}

\newcommand{\abs}[1]{\left\lvert #1 \right\rvert}
\newcommand{\norm}[1]{\left\lVert #1 \right\rVert}

\newcommand{\R}{\mathbb{R}}
\newcommand{\C}{\mathbb{C}}

\renewcommand{\d}{\partial}
\newcommand{\db}{\overline{\partial}}

\title[Radiating and non-radiating sources in elasticity]{Radiating and non-radiating sources in elasticity}

\author[E. Bl{\aa}sten]{Eemeli Bl{\aa}sten}
\address{Jockey Club Institute for Advanced Study, Hong Kong University of Science and Technology, Hong Kong SAR.}
\email{eemeli.blasten@iki.fi}

\author[Y.-H. Lin]{Yi-Hsuan Lin}
\address{Institute for Advanced Study, Hong Kong University of Science and Technology, Hong Kong \& Department of Mathematics and Statistics, University of Jyv\"askyl\"a, Finland}
\email{yihsuanlin3@gmail.com}

\begin{document}

\begin{abstract}
  In this work, we study the inverse source problem of a fixed
  frequency for the Navier's equation.  We investigate nonradiating
  external forces. If the support of such a force has a convex or
  non-convex corner or edge on their boundary, the force must be
  vanishing there. The vanishing property at corners and edges holds
  also for sufficiently smooth transmission eigenfunctions in
  elasticity. The idea originates from the enclosure method: an energy
  identity and a new type exponential solutions for the Navier's
  equation.
  

\end{abstract}

\maketitle

\tableofcontents

\section{Introduction}\label{Section 1}

\subsection{Motivation and background}

The inverse source problem is an important research topic in
scattering theory. Its goal is to determine the shape of unknown
sources by measuring the radiated wave patterns in the far-field. This
problem is motivated by several scientific and industrial areas, such
as medical imaging \cite{fokas2004unique} and photo-acoustic
tomography \cite{arridge1999optical}. A classical example outside of
elasticity is to utilize electric or magnetic fields on the periphery
of the human body, such as in non-invasive brain image
reconstruction. Mathematically, the inverse source problem for
acoustic, electromagnetic and elastic waves has been studied widely by
many researchers
\cite{albanese2006inverse,bao2002inverse,el2011inverse,bao2014inverse,de2017reconstruction,Ikehata1999,li2011inverse}. Furthermore,
the inverse source problem can also be regarded as a basic
mathematical method to study diffusion-based optical tomography, lidar
imaging and fluorescence microscopy.

The inverse source problem of a fixed frequency is ill-posed and its
solution requires a-priori knowledge. These types of problems have
picked interest before too, but for different equations. The umbrella
term \emph{Schiffer's problem} \cite{CK} captures the essence of these
problems: from the knowledge of one far-field pattern, or the Cauchy
data of one solution, determine the shape of a scatterer. \cite{FI}
showed the unique determination of penetrable inclusions for the
conductivity equation by the Cauchy boundary data of one solution,
with reconstruction via the enclosure method appearing in
\cite{Ikehata2000}. In \cite{Ikehata1999}, the enclosure method is
used to reconstruct the the convex hull of a polyhedral source or
potential for the Helmholtz equation given a single measurement. In
\cite{ikehata2009extracting}, the authors adapted the enclosure method
for inverse obstacle problem by using a single set of the Cauchy data
to the Navier's equation. Furthermore, in a slightly different
context, the Schiffer's problem was studied in a nonlocal setup, see
\cite{cao2017simultaneously}.

A new approach to this problem began from the realization that an
acoustic potential that has a corner jump would scatter any incident
wave at any wavenumbers \cite{BPS}. This applies to the potential
scattering setting without a source. However trying to apply it to
source scattering produces an integration by parts formula used in the
enclosure method \cite{Ikehata1999}. The corner scattering method was
then extended and used for the shape-determination of polyhedral
potential scatterers \cite{HSV}, and to sources in \cite{Bsource}. The
latter used a new type of complex geometrical optics solution and the
enclosure method. These techniques also led to the surprising
discovery that interior transmission eigenfunctions \cite{CKP} vanish
at convex corners \cite{BLvanishing,BLvanishingNum}. For more corner
and edge scattering results, we also refer readers to
\cite{elschner2015corners,elschner2018acoustic}. In this paper we
extend the single measurement enclosure method technique to the
elastic setting.

\subsection{The mathematical formulation of elastic waves}
%$\bm{C}=(C_{ijk\ell})_{1\leq i,j,k,\ell \leq n}$ be a constant isotropic elastic tensor, more specifically, $C_{ijk\ell}$ can be interpreted as 
%\begin{align}
%C_{ijk\ell}:=\lambda \delta_{ij}\delta_{k\ell}+\mu (\delta_{ik}\delta_{j\ell}+\delta_{i\ell}\delta_{jk}),\ 1\leq i,j,k,\ell \leq n, \text{ for }n=2,3,
%\end{align}
%where 

%and $\delta_{ij}$ stands for the Kronecker delta. It is easy to see that the elastic tensor $C_{ijk\ell}$ satisfies the major and minor symmetric conditions, i.e.,
%$$
%C_{ijk\ell}=C_{k\ell ij}=C_{jik\ell} \text{ for all }1\leq i,j,k,\ell\leq n.
%$$
Let $\lambda$, $\mu$ be the Lam\'e constants satisfying the following strong convexity condition 
\begin{align}\label{strong convexity condition}
	\mu >0 \text{ and }n\lambda+2\mu >0, \text{ for }n=2,3.
\end{align}
%$$
%\sum _{i,j,k,\ell }^d C_{ijk\ell}a_{ij}a_{k\ell}\geq c_0\sum_{i,j=1}^d a_{ij}^2,
%$$
%for any symmetric matrix $(a_{ij})$ and for some universal constant $c_0>0$. 
Let $\bm{f}\in \C^n$ be an external force, which is assumed to be compactly supported. More specifically, the function $\bm{f}=\chi_\Omega \bm\varphi$, where $\chi_\Omega$ is the characteristic function of a bounded Lipschitz domain $\Omega$ in $\R^n$ and $\bm{\varphi}\in L^\infty (\R^n;\C^n)$. 
Given an angular frequency $\omega >0$, let $\bm{u}(x)=(u_{\ell}(x))_{\ell=1}^{n}$ be the displacement vector field. Then the time-harmonic elastic system is
\begin{align}\label{elasticity system}
\lambda \Delta \bm{u} + (\lambda+\mu)\nabla \nabla\cdot \bm{u}+\omega^2 \bm{u}=\bm{f} \text{ in }\R^n.
\end{align}

Via the well-known Helmholtz decomposition in $\mathbb{R}^{n}\backslash\overline{\Omega}$,
one can see that the scattered field can be decomposed as
\[
\bm{u}=\bm{u}_{p}+\bm{u}_{s} \text{ in }\R^n \setminus \overline{\Omega},
\]
with 
\[
\bm{u}_{p}=-\dfrac{1}{\omega_{p}^{2}}\nabla(\nabla \cdot \bm{u})\mbox{ and } \bm{u}_{s}=\dfrac{1}{\omega_{s}^{2}}\mbox{rot}(\mbox{rot}\bm u),
\]
where $\omega_p$ and $\omega_s$ are the compressional and shear wave numbers, respectively, which are given by $$\omega_{p}=\dfrac{\omega}{\sqrt{\lambda+2\mu}} \text{ and } 
\omega_{s}=\dfrac{\omega}{\sqrt{\mu}}. $$ Above $\mbox{rot}=\nabla^\perp$ represents $\frac{\pi}{2}$ clockwise rotation of the gradient   when $n=2$, and $\mbox{rot}=\nabla\times$ stands for the curl operator when $n=3$. The vector fields $\bm u_{p}$ and $\bm u_{s}$ are
called the compressional and shear parts of
the scattered vector field $\bm u$, respectively. In addition, recall that $\bm{f}=0$ in $\R^n \setminus\overline{\Omega}$. Then $\bm{u}_p$ and $\bm{u}_s$ satisfy
the Helmholtz equation 
\begin{align}\label{Two Helmholtz equations}
\begin{split}
&(\Delta+\omega_{p}^{2})\bm{u}_{p}=0\mbox{ and }\mbox{rot}\bm{u}_{p}=0\mbox{ in }\mathbb{R}^{n}\backslash\overline{\Omega},\\
&(\Delta+\omega_{s}^{2})\bm{u}_{s}=0\mbox{ and }\nabla \cdot\bm{u}_{s}=0\mbox{ in }\mathbb{R}^{n}\backslash\overline{\Omega}.
\end{split}
\end{align}
Therefore, for the elastic scattering problem of \cref{elasticity system}, we need to pose the \textit{Kupradze radiation condition} 
\begin{equation}
\lim_{r\to\infty}\left(\dfrac{\partial \bm{u}_{p}}{\partial r}-i\omega_{p}\bm{u}_{p}\right)=0\mbox{ and }\lim_{r\to\infty}\left(\dfrac{\partial \bm{u}_{s}}{\partial r}-i\omega_{s}\bm{u}_{s}\right)=0,\quad r=|x|,\label{Kupradze radiation condition}
\end{equation}
uniformly in all directions $\widehat{x}={x}/{|x|}$. 
%In this paper, the notation $\cdot$ is the standard inner product in $\mathbb R^d$ and $:$ stands for the contraction between two tensors.
Moreover, one can also expand the functions $\bm{u}_s$ and $\bm{u}_p$ as 
\begin{align}\label{elastic far fields}
\begin{split}
& \bm{u}_s(x)=\dfrac{1}{4\pi }\dfrac{e^{i\omega_s|x|}}{|x|^{\frac{n-1}{2}}}\bm{u}_s ^\infty (\widehat{x})+O\left(|x|^{-\frac{n+1}{2}}\right) \text{ as }|x|\to \infty,\\
& \bm{u}_p(x)=\dfrac{1}{4\pi }\dfrac{e^{i\omega_p|x|}}{|x|^{\frac{n-1}{2}}}\bm{u}_p ^\infty (\widehat{x})+O\left(|x|^{-\frac{n+1}{2}}\right) \text{ as }|x|\to \infty,
\end{split}
\end{align}
for $n=2,3$, where $\bm{u}_s^\infty $ and $\bm{u}_p^\infty$ denote the \emph{transversal} and \emph{longitudinal elastic far fields} radiated by the source $\bm{f}$. Furthermore, $\bm{u}_s^\infty$ and $\bm{u}_p^\infty$ can be explicitly represented by 
$$
\bm{u}_s^\infty (\bm{e})= \Pi_{\bm{e}^\perp}\left(\int_{\R^n}e^{-i\omega_s\bm{e}\cdot y}\bm{f}(y)dy\right), \ 
\bm{u}_p^\infty (\bm{e})= \Pi_{\bm{e}}\left(\int_{\R^n}e^{-i\omega_p\bm{e}\cdot y}\bm{f}(y)dy\right),
$$ for any unit vector $\bm{e}\in \mathbb{S}^{n-1}$, where
$\Pi_{\bm{e}}$ is the projection operator with respect to $\bm{e}$.
Notice that the vector fields $\bm{u}_s^\infty$ and $\bm{u}_p^\infty$
are the tangential and the normal components of the Fourier transform
of $\bm{f}$ evaluated on $\mathbb{S}^{n-1}$. Note that the elastic far
fields \eqref{elastic far fields} of the Navier's equation are derived
using the Helmholtz decomposition of \cref{elasticity system} and the
far-field patterns for the Helmholtz equations of \eqref{Two Helmholtz
  equations}, which is allowed by the radiation conditions of
\cref{Kupradze radiation condition}. For a more detailed discussion,
we refer readers to
\cite{griesmaier2014far,griesmaier2018uncertainty,hahner1998acoustic}.

It is known that for a given source function $\bm{f}\in L^\infty (\Omega;\C^n)\subset L^2(\Omega;\C^n)$, the scattering problem of \cref{elasticity system,Kupradze radiation condition} has a unique solution (see \cite{bao2017inverse} for instance):

$$
\bm{u}(x;\omega)=\int_\Omega \bm{G}(x,y;\omega)\cdot \bm{f}(y)dy,
$$
where $\bm{G}(x,y;\omega)\in \C^{n\times n}$ is the Green's tensor for the Navier's equation in \cref{elasticity system}. More precisely, the Green's tensor can be expressed as 
$$
\bm{G}(x,y;\omega)=\dfrac{1}{\mu} G_n(x,y;\omega_s)I_n+\dfrac{1}{\omega^2}\nabla_x \nabla_x^\perp \left(G_n(x,y;\omega_s)-G_n(x,y;\omega_p)\right),
$$
where $I_n$ denotes an $n\times n$ identity matrix and 
$$
G_n(x,y;\omega)= \begin{cases}
\dfrac{i}{4}H_0 ^{(1)}(\omega |x-y|) & \text{ when }n=2, \\
\dfrac{1}{4\pi}\dfrac{e^{i\omega |x-y|}}{|x-y|} & \text{ when }n=3,
\end{cases}
$$
is the fundamental solution for the Helmholtz equation. Here $H_0 ^{(1)}$ is the Hankel function of the first kind with order zero.

We prove the following three theorems. The first one is for the inverse source problem. It states that if an external force is applied to a region having a corner or edge on an elastic body, it creates a propagating elastic wave at \emph{any} wavenumber.

\begin{theorem}\label{thm1}
  Let $\bm{f} = \chi_\Omega \bm{\varphi}$ for a bounded domain
  $\Omega\subset\R^n$, $n\in\{2,3\}$ and bounded vector function
  $\bm{\varphi} \in L^\infty(\R^n)$. Let $\omega,\mu>0$,
  $n\lambda+2\mu>0$ and $\bm{u}\in H^2_{loc}(\R^n)$ satisfy
  \cref{elasticity system} and the radiation condition of
  \cref{Kupradze radiation condition}.

  Assume that $\Omega$ has a corner (2D) or an edge (3D) that can be
  connected to infinity by a path in $\R^n\setminus \overline\Omega$,
  and that $\bm{\varphi}$ is H\"older-continuous near it. If $\bm{u}$
  has zero far-field pattern, then $\bm{\varphi}=0$ on the corner or
  edge, i.e. $\bm{\varphi}$ is the zero vector. In other words, $f$
  has no jumps at these locations.
\end{theorem}

\begin{remark*}
Recall that by the Helmholtz decomposition, one can reduce the
Navier's equation of \eqref{elasticity system} into two Helmholtz
equations as in \eqref{Two Helmholtz equations} outside of
$\Omega$. One might wonder if it were possible to prove \cref{thm1}
using the same decomposition for the source term and using a version
of \cref{thm1} proven in \cite{Bsource} that holds for the two
Helmholtz equations of \eqref{Two Helmholtz equations}
separately. After all, by the Helmholtz decomposition, one could split
$f=f_p+f_s$ and rewrite \cref{elasticity system} into
\[
(\Delta +\omega_p ^2 )u_p=f_p, \text{ and }(\Delta +\omega_s ^2
)u_s=f_s \text{ in }\R^n.
\]
However there is a problem of smoothness. \cref{thm1} assumes only
that the function $f$ is H\"older-continuous near a corner. This does
not guarantee that both of $f_p$ and $f_s$ are H\"older-continuous
near the same corner. This prevents the use of
\cite[Theorem~1.1]{Bsource} to derive the same conclusion for the
Navier's equation.
\end{remark*}

In order to prove \cref{thm1}, we  need to construct a new type exponential solution for the Navier's equation in the plane. The classical exponential solutions for the Navier's equation is called the complex geometrical optics (CGO) solutions. As a matter of fact, the CGO solutions \cite{imanuvilov2015global} were utilized to show global uniqueness for the inverse problem of the isotropic elasticity system with infinite measurements. In \cite{imanuvilov2015global}, the authors introduced CGO solutions for the isotropic elasticity system by decoupling the system into weakly coupled systems. Their principal part is an exponential of quadratic complex-valued function.

\smallskip
The second theorem shows source shape and boundary value
determination from a single far-field pattern.

\begin{theorem}\label{thm2}
  Let $n\in\{2,3\}$ and $\Omega,\Omega'\subset\R^n$ be bounded convex
  polyhedral domains. Let $\bm{\varphi},\bm{\varphi'} \in
  C^\alpha(\R^n)$, for some $\alpha\in (0,1)$ and have nonzero value on
  $\partial\Omega, \partial\Omega'$.

  Define $\bm{f} = \chi_\Omega\bm{\varphi}$, $\bm{f'} =
  \chi_{\Omega'}\bm{\varphi'}$. Let $\omega,\mu>0$, $n\lambda+2\mu>0$
  and $\bm{u},\bm{u'}\in H^2_{loc}(\R^n)$ have elastic sources
  $\bm{f},\bm{f'}$. In other words they satisfy \cref{elasticity
    system} with the radiation condition of \cref{Kupradze radiation
    condition}.

  If $\bm{u}$ and $\bm{u'}$ have the same far-field pattern then
  $\Omega=\Omega'$ and $\bm{\varphi}=\bm{\varphi'}$ at each of their
  vertices and in three dimensions, edges.
\end{theorem}

Strictly speaking, we do not need to have $\bm{\varphi},\bm{\varphi'}$
H\"older-continuous everywhere or non-vanishing on the whole boundary
for the unique determination of the shape. Near the corners is
enough. In three dimensions, if they are in addition non-vanishing and
H\"older-continuous near the edges, then we can also deduce that
$\bm{\varphi}=\bm{\varphi'}$ on them.

\smallskip
\begin{definition}[Interior transmission eigenfunctions]
A pair $(\bm{v},\bm{w}) \in L^2(\Omega)\times L^2(\Omega)$ is called
\emph{interior transmission eigenfunctions} for the Navier equations
with density $V\in L^\infty(\Omega)$ at the \emph{interior
	transmission eigenvalue} $\omega\in\R_+$ if
\begin{align}\label{Equ of interior transmission eigenvalues}
\begin{cases}
\lambda \Delta \bm{w} + (\lambda+\mu)\nabla \nabla\cdot
\bm{w}+\omega^2 \bm{w} = 0, \\ 
\lambda \Delta \bm{v} +
(\lambda+\mu)\nabla \nabla\cdot \bm{v}+\omega^2(1+V) \bm{v} = 0,
\end{cases}
\end{align}
and $\bm{v}-\bm{w} \in H^2(\Omega)$ with $\bm{v}=\bm{w}$ and
$\bm{T_\nu v} = \bm{T_\nu w}$ on $\partial\Omega$. Nothing is imposed
on the boundary values of $\bm{v},\bm{w}$ individually.
\end{definition}

Above $\bm{T_\nu}$ is the boundary tration operator.
  \begin{definition} \label{boundary_traction}
    The \emph{boundary traction} operator $\bm{T_\nu}$ is defined as
    follows. In the two-dimensional case it is
    \begin{align*}
      \bm{T}_{\bm{\nu}}\bm{u}= 2\mu\dfrac{\partial
        \bm{u}}{\partial\bm{\nu}}+\lambda\bm{\nu}\nabla\cdot
      \bm{u}+\mu\bm{\nu}^{\perp}(\partial_{2}u_{1}-\partial_{1}u_{2}),
    \end{align*}
    where $\bm{\nu}=(\nu_1,\nu_2)$ is a unit outer normal on $\p
    \Omega$ and $\bm{\nu}^{\perp}:=(-\nu_2,\nu_1)$. In the three
    dimensional case,
    \begin{align*}
      \bm{T}_{\bm{\nu}}\bm{u}=2\mu\dfrac{\partial
        \bm{u}}{\partial\bm{\nu}}+\lambda\bm{\nu}\nabla\cdot
      \bm{u}+\mu\bm{\nu}\times(\nabla\times \bm{u}),
    \end{align*}
    where $\bm{\nu}=(\nu_1,\nu_2,\nu_3)$.
  \end{definition}

The interior transmission problem is a well studied problem in inverse
scattering theory. In particular the sampling method for solving the
inverse scattering problem fails at wavenumbers that are transmission
eigenvalues \cite{CK}. Notable results, \cite{CKP, RS, PS, CGH} and
the recent survey \cite{CHinBook}, have so far focused mostly on the
spectral properties of the transmission eigenvalues. Only recently,
e.g. \cite{BP,Robbiano,LV}, have results about the eigenfunctions
themselves started surfacing, in the form of completeness of the
eigenfunctions in some sense. These past results are mostly in the
context of acoustic scattering for the Helmholtz
equation. \cite{elasticITP} considers the fundamental properties in
the context of elasticity, and shows the existence and discreteness of
the transmission eigenvalues.

Recently in \cite{BLvanishing,BLvanishingNum,Bsource} it was shown
that under given smoothness and geometric assumptions the transmission
eigenfunctions for the Helmholtz equation vanish at convex corners of
the domain $\Omega$. We show the same conclusion for the interior
transmission problem for an elastic material with varying density.
\begin{theorem} \label{ITP}
  Let $n\in \{2,3\}$ and $\Omega\subset\R^n$ be a bounded domain. Let
  $V\in L^\infty(\Omega)$ be the material density, and $\mu>0$,
  $n\lambda+2\mu>0$ be constant Lam\'e parameters. Assume that
  $\omega>0$ is an interior transmission eigenvalue and
  $\bm{v},\bm{w}\in L^2(\Omega)$ are the corresponding transmission
  eigenfunctions defined by \cref{Equ of interior transmission eigenvalues}.

  Let $x_c$ be any vertex or edge of $\p \Omega$ around which $V$ and
  either one of $\bm{v},\bm{w}$ are $C^\alpha$ smooth, for some 
  $\alpha\in (0,1)$. Then so is the other, and $\bm{v}(x_c)=\bm{w}(x_c)=0$ if
  $V(x_c) \neq 0$.
\end{theorem}

The paper is structured as follows. In \cref{Section 2}, we construct new type exponential solution for the two-dimensional isotropic elastic case and develop the integration by parts formula for the domain with corners. In \cref{Section 3}, we discuss the corner scattering in a plane, and we use the dimensional reduction technique to solve the three-dimensional case. In addition, we prove our theorems in \cref{Section 4}. Finally, in \cref{appendix}, we show that there is a nonradiating source $\bm{f}$ that generates elastic far fields that vanish in all directions.

\section{Exponential solutions of the Navier's equation in the plane}\label{Section 2}

We write $\bm{\mathcal{L}}:=\lambda \Delta +(\lambda+\mu)\nabla (\nabla \cdot)$ for the second order elliptic operator. If $\bm{u}$ solves the elasticity system of \cref{elasticity system} then $\bm{\mathcal{L}}u=\bm{f}$ in $\R^n$ for $n=2,3$, where $\bm{f}$ is the given compactly supported bounded source function. Next, let $\bm{v}$ be a solution of $\bm{\mathcal{L}}\bm{v}=0$. Via the standard regularity theory for the second order elliptic system (see \cite[Theorem 4.16]{mclean2000strongly} for instance), one can obtain that $u,v$ are $H^2_{loc}(\R^n)$ functions. Then
integrating by parts twice yields that 
\begin{align} \label{easyIntId}
\int_\Omega \bm{f}\cdot\bm{v} dx = \int_\Omega (\bm{\mathcal{L}}\bm{u})\cdot \bm{v}dx = \int_{\p \Omega} \left[(\bm{T}_{\bm{\nu}}\bm{u})\cdot\bm{v}- (\bm{T}_{\bm{\nu}}\bm{v})\cdot\bm{u} \right]dS,
\end{align}
where $\bm{T}_{\bm{\nu}}$ stands for the boundary traction operator (\cref{boundary_traction}) of \cref{elasticity system}.

%Well-posedness of the direct problem
In two dimensions, note that the system of \cref{elasticity system} can be expressed componentwise as 
\begin{equation} \label{sourceElasticity}
    \bm{\mathcal{L}} \bm{u} = \left(
    \begin{matrix}
      \lambda \Delta + (\lambda+\mu)\partial_1^2 & (\lambda+\mu)
      \partial_1\partial_2 \\ (\lambda+\mu) \partial_1 \partial_2 &
      \lambda \Delta + (\lambda+\mu)\partial_2^2
    \end{matrix}
    \right) \bm{u} = \bm{f} \text{ in }\R^2.
  \end{equation}
From now on we identify $\R^2$ with the complex plane $\C$, and we have the following lemma.
  %in $\R^2$ or a domain $\Omega$, and write conditions when this has a unique solution (radiation condition, assumptions on the boundary $\partial\Omega$?) 
\begin{lemma} \label{vDef}
  Let $\Omega\subset\C$ such that $\Omega \cap (\R_-\cup\{0\}) =
  \emptyset$. Let
  \begin{equation}\label{exponential solution v}
    \bm{v}(x) = \left(
    \begin{matrix}
      \exp(-s\sqrt{z})\\ i\exp(-s\sqrt{z})
    \end{matrix}
    \right)
  \end{equation}
  where $z = x_1+ix_2$ and $s\in\R_+$. The complex square root is
  defined as
  \begin{equation} \label{sqrt}
    \sqrt{z} = \sqrt{\abs{z}} \left( \cos \frac\theta2 + i \sin
    \frac\theta2 \right)
  \end{equation}
  where $ -\pi < \theta \leq \pi$ is the argument of $z$. Then
  $\bm{v}$ satisfies $\bm{\mathcal L}\bm{v}=0$ in $\Omega$.
\end{lemma}
\begin{proof}
  The complex square root is complex analytic in $\C$ except at the
  branch cut composed of the origin and the negative real axis. Thus
  the components of $\bm{v}=(v_1, v_2)^T$ are complex analytic in
  $\Omega$. Let us rewrite the differential operator in
  \cref{sourceElasticity} using the complex derivatives
  \[
  \d = \frac12 \big(\partial_1 - i \partial_2\big), \qquad \db =
  \frac12 \big(\partial_1 + i \partial_2\big).
  \]
  These imply that $\Delta = 4\d\db$ and that $\partial_1 = \d+\db$,
  $\partial_2 = i(\d-\db)$. Hence the abovementioned operator becomes
  \begin{equation}
    \bm{\mathcal{L}} = \left(
    \begin{matrix}
      4\lambda \d\db + (\lambda+\mu)(\d^2 + 2 \d\db + \db^2) &
      (\lambda+\mu) i (\d^2 - \db^2) \\ (\lambda+\mu) i (\d^2-\db^2) &
      4\lambda \d\db - (\lambda+\mu)(\d^2 -2\d\db + \db^2)
    \end{matrix}
    \right)
  \end{equation}
  and thus $\bm{\mathcal L} \bm{v} = 0$ if and only if
  \[
  \begin{cases}
    4\lambda\d\db v_1 + (\lambda+\mu)(\d^2 v_1 + 2\d\db v_1 + \db^2
    v_1 + i \d^2 v_2 - i \db^2 v_2) = 0,\\ 4\lambda\d\db v_2 +
    (\lambda+\mu)(i\d^2 v_1 - i\db^2 v_1 - \d^2 v_2 + 2\d\db v_2 -
    \db^2 v_2) = 0.
  \end{cases}
  \]
  Recall that our choice of $\bm{v}$ implies that $\db v_1 = \db v_2 =
  0$ in $\Omega$. Hence the above reduces to
  \[
  \d^2 v_1 + i \d^2 v_2 = 0, \qquad i\d^2 v_1 - \d^2 v_2 = 0
  \]
  which is just equivalent to $\d^2 v_1 + i \d^2 v_2 = 0$. But this is
  true since $v_1 + i v_2 = 0$. Hence $\bm{\mathcal L} \bm{v} = 0$ in
  $\Omega$.
\end{proof}

The following lemma provides the integration by parts formula around the corner, where $\bm{v}$ is not smooth.

\begin{lemma} \label{vIntId}
  Let $\bm{v}:\R^2\to\C^2$ be defined by \eqref{exponential solution v} from \cref{vDef} and
  \[
  \mathcal{K} = \{x\in\R^2 \mid x\neq0,\, \theta_m<\arg(x_1+ix_2)<\theta_M \}
  \]
  for given angles $-\pi<\theta_m<\theta_M<\pi$. Assume that
  $\bm{u}\in H^2(\mathcal{K}\cap B;\C^2)$ where $B=B(0,h)$ for some $h>0$. Moreover
  let $\bm{u}=\bm{T_\nu} \bm{u} = 0$ in $B\cap \partial \mathcal{K}$. Then
  \begin{equation}\label{corner integration by parts}
    \int_{\mathcal{K}\cap B} \bm{v} \cdot \bm{\mathcal L}\bm{u}\, dx =
    \int_{\mathcal{K}\cap\partial B} [(\bm{T_\nu u})\cdot\bm{v} - (\bm{T_\nu
        v})\cdot\bm{u}] dS.
  \end{equation}
\end{lemma}
\begin{proof}
  Let $\mathcal{K}_\varepsilon := (\mathcal{K}\cap B) \setminus B(0,\varepsilon)$ for
  $0<\varepsilon<h$. Note that $\abs{\bm{v}}$ is bounded and
  $\abs{\bm{\mathcal Lu}}\in L^2$, and that $\bm{v}$ is smooth and
  $\bm{\mathcal L}\bm{v}=0$ in the closure of $\mathcal{K}_\varepsilon$. Hence we have
  \begin{align}\label{int by part 1}
  \int_{\mathcal{K}\cap B} \bm{v} \cdot \bm{\mathcal L u}\, dx =
  \lim_{\varepsilon\to0} \int_{\mathcal K_\varepsilon} \bm{v} \cdot
  \bm{\mathcal L u}\, dx = \lim_{\varepsilon\to 0} \int_{\d
  	\mathcal{K}_\varepsilon} [(\bm{T}_{\bm{\nu}}\bm{u})\cdot\bm{v}-
  (\bm{T}_{\bm{\nu}}\bm{v})\cdot\bm{u}] dS
  \end{align}
  by \cref{easyIntId}. Recall that $\bm{u} = \bm{T_\nu u} = 0$ on
  $B\cap\d \mathcal{K}$. Thus the boundary integral of \cref{int by part 1} is equal to
  \begin{align}\label{int by part 2}
  \int_{\mathcal{K} \cap \d B} [(\bm{T}_{\bm{\nu}}\bm{u})\cdot\bm{v}-
  (\bm{T}_{\bm{\nu}}\bm{v})\cdot\bm{u}] dS + \int_{\mathcal{K} \cap
  	S(0,\varepsilon)} [(\bm{T}_{\bm{\nu}}\bm{u})\cdot\bm{v}-
  (\bm{T}_{\bm{\nu}}\bm{v})\cdot\bm{u}] dS,
  \end{align}
  where $S(0,\epsilon)=\p B(0,\epsilon)$, and so our task is to show that the last integral in \cref{int by part 2} vanishes as $\varepsilon\to0$.

  Notice that the normal derivative is $\bm{\nu}(x) =
  -x/\abs{x}$ on  $S(0,\varepsilon)$. Recall also that $\d_1 = \d+\db$ and $\d_2 =
  i(\d-\db)$, and moreover that $\bm{v}$ is complex analytic in a
  neighbourhood of $S(0,\varepsilon)$. Hence $\d_1\bm{v} = \d\bm{v}$
  and $\d_2\bm{v} = i\d\bm{v}$. Concluding, the boundary traction
  becomes
  \begin{align*}
    \bm{T_\nu v} &= 2\mu (\nu_1 \d_1 + \nu_2 \d_2) \bm{v} + \lambda
    \bm{\nu} (\d_1 v_1 + \d_2 v_2) + \mu \bm{\nu}^\perp (\d_2 v_1 -
    \d_1 v_2) \\ &= 2\mu (\nu_1 + i\nu_2) \d \bm{v} + \lambda \bm{\nu}
    (\d v_1 + i \d v_2) + \mu \bm{\nu}^\perp (i\d v_1 - \d v_2),
  \end{align*}
  however note that $v_2 = i v_1$ so the divergence and curl terms
  vanish. Moreover $\nu_1+i\nu_2 = -z/\abs{z}$ with the identification
  $z = x_1+ix_2$. Hence
  \begin{equation} \label{vTraction}
    \bm{T_\nu v} = - 2\mu \frac{z}{\abs{z}} \d \bm{v} = s \mu
    \frac{\sqrt{z}}{\abs{z}} \bm{v}
  \end{equation}
  because $\partial \exp(-s\sqrt{z}) = \frac{-s \exp(-s\sqrt{z})}{2\sqrt{z}}$ outside of the branch cut.

  We have $\abs{\bm{v}(x)} = \sqrt{2} \exp(-s\sqrt{\abs{x}}\cos
  \theta/2) \leq\sqrt{2}$ where $\theta$ is the argument of
  $x_1+ix_2$. By \cref{vTraction}
  \[
  \abs{\bm{T_\nu v}} = \abs{s \mu \frac{\sqrt{z}}{\abs{z}} \bm{v}} = s
  \mu \varepsilon^{-1/2} \abs{\bm{v}} \leq \sqrt{2} s \mu
  \varepsilon^{-1/2}
  \]
  on $S(0,\varepsilon)$. Note also that $H^2(\mathcal{K}\cap B)$-functions embed
  continuously into uniformly bounded and continuous functions in two
  dimensions. Hence
  \[
  \abs{\bm{u}(x)} \leq C \norm{\bm{u}}_{H^2(\mathcal{K}\cap B)}
  \]
  for $x\in \mathcal{K}\cap B$ and in particular for $x\in
  S(0,\varepsilon)$. This shows that
  \[
  \abs{\int_{\mathcal{K}\cap S(0,\varepsilon)} (\bm{T_\nu v})\cdot \bm{u}(x) dS
  } \leq C \norm{\bm{u}}_{H^2(\mathcal{K}\cap B)} \sqrt{2} s \varepsilon^{-1/2}
  (\theta_M-\theta_m)\varepsilon \to 0
  \]
  as $\varepsilon\to0$.

  For the remaining term, we can estimate $\abs{\bm{v}}\leq\sqrt{2}$
  and use the Cauchy--Schwarz inequality to get
  \begin{equation}\label{sphereIntegral}
    \abs{\int_{\mathcal{K}\cap S(0,\varepsilon)} (\bm{T_\nu u})\cdot\bm{v} dS }
    \leq \sqrt{(\theta_M-\theta_m)\varepsilon} \norm{ \bm{T_\nu u}
    }_{L^2(\mathcal{K}\cap S(0,\varepsilon))} \sqrt{2}.
  \end{equation}
  Denote $\bm{g}(x) = \bm{T_\nu u}(x)$. Then $\bm{g}\in H^1(K\cap B)$
  and $\norm{\bm{g}}_{H^1} \leq C(\lambda+\mu) \norm{\bm{u}}_{H^2}$ in
  any given open set. Let $\bm{g}_\varepsilon(y) = \bm{g}(\varepsilon
  y)$. Then $dS(x) = \varepsilon dS(y)$ so
  \[
  \norm{\bm{g}}_{L^2(\mathcal{K}\cap S(0,\varepsilon))} = \sqrt{\varepsilon}
  \norm{\bm{g}_\varepsilon}_{L^2(\mathcal{K}\cap S(0,1))} \leq C
  \sqrt{\varepsilon} \norm{\bm{g}_\varepsilon}_{H^1(\mathcal{K}\cap B(0,1))}
  \]
  by the trace theorem, and $C>0$ is independent of $\varepsilon$.
  However $dx = \varepsilon^2 dy$, and $\d_{y_j} \bm{g}_\varepsilon(y)
  = \d_{y_j} (\bm{g}(\varepsilon y)) = \varepsilon (\d_j \bm{g})
  (\varepsilon y)$ so
  \begin{align*}
    &\norm{\bm{g}_\varepsilon}_{L^2(\mathcal{K}\cap B(0,1))} = \varepsilon^{-1}
    \norm{\bm{g}}_{L^2(\mathcal{K}\cap B(0,\varepsilon))}, \\ &\norm{\d_j
      \bm{g}_\varepsilon}_{L^2(\mathcal{K}\cap B(0,1))} = \norm{\d_j
      \bm{g}}_{L^2(\mathcal{K}\cap B(0,\varepsilon))},
  \end{align*}
  in other words $\norm{\bm{g}_\varepsilon}_{H^1(K\cap B(0,1))} \leq
  \varepsilon^{-1} \norm{\bm{g}}_{H^1(K\cap B(0,\varepsilon))}$. This
  implies
  \[
  \norm{\bm{g}}_{L^2(\mathcal{K}\cap S(0,\varepsilon))} \leq C
  \varepsilon^{-1/2} \norm{\bm{g}}_{H^1(\mathcal{K}\cap B(0,\varepsilon))} \leq
  C (\lambda+\mu) \varepsilon^{-1/2} \norm{\bm{u}}_{H^2(\mathcal{K}\cap
    B(0,\varepsilon))}
  \]
  where $C$ is independent of $\varepsilon$. Combining these with
  \eqref{sphereIntegral} gives
  \[
    \abs{\int_{\mathcal{K}\cap S(0,\varepsilon)} (\bm{T_\nu u})\cdot\bm{v} dS } \leq C (\lambda+\mu) \sqrt{\theta_M-\theta_m}
  \norm{\bm{u}}_{H^2(\mathcal{K}\cap B(0,\varepsilon))}
  \]
  which tends to zero when $\varepsilon\to0$ because
  $\norm{\bm{u}}_{H^2(K\cap B)}$ is finite. This proves \cref{corner integration by parts}.
\end{proof}

\section{Corner scattering}\label{Section 3}

\begin{proposition}\label{2DupperAndLower}
  Let $\bm{v}:\R^2\to\C$ be the function given in \cref{vDef} and define the open
  sector
  \begin{equation*} 
    \mathcal{K} = \{x\in\R^2 \mid x\neq0,\, \theta_m < \arg(x_1+ix_2) < \theta_M
    \}
  \end{equation*}
  for angles satisfying $-\pi < \theta_m < \theta_M < \pi$. Then
  \begin{equation*}
    \int_\mathcal{K} v_1(x) dx = 6i ( e^{{-2}{\theta_M}i} - e^{{-2}{\theta_m}i}
    ) s^{-4}.
  \end{equation*}
  In addition for $\alpha,h>0$ and $j\in\{1,2\}$ we have the upper
  bounds
  \begin{equation*}
    \int_\mathcal{K} \abs{v_j(x)} \abs{x}^\alpha dx \leq
    \frac{2(\theta_M-\theta_m)
      \Gamma(2\alpha+4)}{\delta_\mathcal{K}^{2\alpha+4}} s^{-2\alpha-4}
  \end{equation*}
  and
  \begin{equation*}
    \int_{\mathcal{K} \setminus B(0,h)} \abs{v_j(x)} dx \leq
    \frac{6(\theta_M-\theta_m)}{\delta_\mathcal{K}^4} s^{-4} e^{-\delta_\mathcal{K} s
      \sqrt{h}/2}.
  \end{equation*}
  where $\delta_\mathcal{K} = \min_{\theta_m<\theta<\theta_M} \cos(\theta/2)$ is a positive constant.
\end{proposition}
\begin{proof}
  This result is in \cite[Section 2]{Bsource}, however note that the
  parameter $s$ is outside of the square root in this article.
\end{proof}

\begin{proposition} \label{2Dprop}
  Let $\Omega\subset\R^2$ be a bounded domain and define the cone
  \begin{align}\label{cone}
  \mathcal{K} = \{x\in\R^2 \mid x\neq0,\, \theta_m < \arg(x_1+ix_2) < \theta_M
  \}
  \end{align}
  with angles $-\pi<\theta_m<\theta_M<\pi$ where $\theta_M \neq
  \theta_m + \pi$. Assume that $0\in\partial\Omega$ is the centre of a
  ball $B$ for which $\Omega\cap B=\mathcal{K}\cap B$.

  Given $\alpha\in (0,1)$ and  $\bm{f}\in
  C^\alpha(\overline{\Omega\cap B})$, let $\bm{u}\in H^2(\Omega\cap B)$ solve
  \[
  \lambda \Delta \bm{u} + (\lambda+\mu)\nabla \nabla\cdot
  \bm{u}+\omega^2 \bm{u} = \bm{f} \text{ in }\Omega\cap B,
  \]
  for some fixed $\omega>0$. If $\bm{u}=0$ and
  $\bm{T_\nu u} = 0$ on $\partial\Omega\cap B$ then $\bm{f}(0) = 0$.
\end{proposition}
\begin{proof}
  Write the equation as $\bm{\mathcal L u} + \omega^2 \bm{u}= \bm{f}$
  where $\bm{\mathcal L}$ is as in \cref{sourceElasticity}. We may
  assume that $\bm{f}$ is real-valued. If not, then split it and
  $\bm{u}$ into their real and imaginary parts $\bm{f}_R, \bm{f}_I,
  \bm{u}_R, \bm{u}_I$ which would then satisfy
  \begin{align*}
    &\bm{\mathcal L} \bm{u}_R + \omega^2 \bm{u}_R = \bm{f}_R,\\
    &\bm{\mathcal L} \bm{u}_I + \omega^2 \bm{u}_I = \bm{f}_I,
  \end{align*}
  with $\bm{u}_R = 0 = \bm{u}_I$, $\bm{T_\nu} \bm{u}_R = 0 =
  \bm{T_\nu} \bm{u}_I$ on $\d\Omega\cap B$. Then once we show that
  $\bm{f}_R(0) = \bm{f}_I(0) = 0$ this would imply the original claim
  of $\bm{f}(0)=0$.

  Let $\bm{v}$ be as in \cref{vDef} and so $\bm{\mathcal L v} = 0$ in
  $\Omega\cap B$. Let $\widetilde{\bm{f}} = \bm{f} - \omega^2 \bm{u}$ and
  keep in mind that $\widetilde{\bm{f}}(0) = \bm{f}(0)$ because $\bm{u}=0$
  on $\d\Omega\cap B$. \cref{vIntId} implies that
  \begin{equation} \label{intId}
    \int_{\mathcal{K}\cap B} \bm{v} \cdot \widetilde{\bm{f}} dx = \int_{\mathcal{K}\cap\d B} [
      (\bm{T_\nu u})\cdot\bm{v} - (\bm{T_\nu v})\cdot\bm{u} ] dS
  \end{equation}
  and moreover the left-hand side of \cref{intId} can be split as
  \begin{equation} \label{threeTerms}
    \int_{\mathcal{K}\cap B} \bm{v} \cdot \widetilde{\bm{f}} dx = \int_\mathcal{K} \bm{v}\cdot\bm{f}(0) dx - \int_{\mathcal{K}\setminus B}
    \bm{v}\cdot\bm{f}(0) dx + \int_{\mathcal{K}\cap B} \bm{v} \cdot \left(
    \widetilde{\bm{f}}-\widetilde{\bm{f}}(0) \right) dx
  \end{equation}
  because $\widetilde{\bm{f}}(0) = \bm{f}(0)$.

  As in the proof of \cref{vIntId}, we have $\bm{T_\nu v} = -s\mu\left(\frac{\sqrt{z}}{\abs{z}}\right) \bm{v}$ on $\mathcal{K}\cap\d B$ and so $\abs{\bm{T_\nu v}} =
  s \mu h^{-1/2} \sqrt{2} \exp(-s\sqrt{h} \cos(\theta/2))$ on $\mathcal{K}\cap \p B$ where
  $h$ is the radius of $B$ and $\theta$ the argument of
  $x_1+ix_2$. Let $\delta_\mathcal{K} := \min_{\theta_m<\theta<\theta_M}
  \cos(\theta/2)$ and it is positive. Hence
  \begin{equation} \label{TvBestimate}
    \abs{\bm{T_\nu v}(x)} \leq C_{\mathcal{K},B,\mu} s \exp(- \delta_\mathcal{K} s
    \sqrt{h})
  \end{equation}
  for some positive constant $C_{\mathcal{K},B,\mu}$ (depending on $\mathcal{K}$, $B$ and $\mu$), when $x\in \mathcal{K}\cap\d B$. By a similar consideration
  \begin{equation} \label{vBestimate}
    \abs{\bm{v}(x)} \leq \sqrt{2} \exp(-\delta_\mathcal{K} s \sqrt{h})
  \end{equation}
  on the same part of the boundary. Thus, \cref{TvBestimate}, \cref{vBestimate}, the Cauchy--Schwarz and trace inequalities give that 
  \begin{align} \label{Bestimate}
    &\abs{\int_{\mathcal{K}\cap\d B} [ (\bm{T_\nu u})\cdot\bm{v} - (\bm{T_\nu
          v})\cdot\bm{u} ] dS} \notag\\ &\qquad \leq
    C_{\mathcal{K},B,\mu,\lambda} \norm{\bm{u}}_{H^2(\mathcal{K}\cap B)} (1+s)
    \exp(-\delta_\mathcal{K} s \sqrt{h})
  \end{align}
  which decays exponentially as $s\to\infty$.

  Let us estimate the three terms in \cref{threeTerms} next. But
  before that, we recall $\widetilde{\bm{f}} = \bm{f} - \omega^2 \bm{u}$,
  and that $\bm{u}\in H^2$ which embeds continuously into $C^\alpha$
  in two dimensions. Hence $\widetilde{\bm{f}} \in
  C^\alpha(\overline{K\cap B})$. Then \cref{2DupperAndLower} gives
  \begin{align*}
    &\int_\mathcal{K} \bm{v}\cdot\bm{f}(0) dx = (f_1(0)+if_2(0)) \int_\mathcal{K} v_1(x)
    dx = C_{\mathcal{K}} (f_1(0)+if_2(0)) s^{-4} \\ &\int_{\mathcal{K}\setminus B}
    \abs{\bm{v}\cdot\bm{f}(0)} dx \leq C'_\mathcal{K} \norm{\bm{f}}_{L^\infty}
    s^{-4} \exp(-\delta_\mathcal{K} s \sqrt{h}/2)\\ & \int_{\mathcal{K}\cap B} \abs{
      \bm{v}\cdot \big( \widetilde{\bm{f}}-\widetilde{\bm{f}}(0) \big)} dx
    \leq \norm{\widetilde{\bm{f}}}_{C^\alpha} \sum_{j=1}^2 \int_{\mathcal{K}\cap B}
    \abs{v_j(x)} \abs{x}^\alpha dx \\ &\qquad \qquad \leq C_{\mathcal{K},\alpha}
    \norm{\widetilde{\bm{f}}}_{C^\alpha} s^{-2\alpha-4}
  \end{align*}
  where $C_\mathcal{K}\neq0$ and the other constants $C'_\mathcal{K}$, $C_{\mathcal{K},\alpha}$ are finite and positive. Multiplying
  \cref{intId,threeTerms} by $s^4$ and letting $s\to\infty$ implies
  that
  \[
  C_\mathcal{K}(f_1(0) + if_2(0)) = 0
  \]
  and since $\bm{f}$ is real-valued and $C_\mathcal{K}\neq0$, that
  $\bm{f}(0)=0$.
\end{proof}

\begin{lemma}[Dimension reduction]\label{Lem dim reduce}
Let $D$ be a locally Lipschitz open set in $\R^2$, $M>0$ and $\alpha \in (0,1)$ be constants. Given $\bm{f}\in C^{\alpha}(\overline{D}\times [-L,L];\C^3)$, let $\bm{u}\in H^2(D\times (-L,L);\C^3)$ be a solution of 
\begin{align*}
\begin{cases}
\bm{\mathcal{L}\bm{u}}(x)=\bm{f}(x), & \text{ for }x=(x',x_3)\in D\times (-L,L), \\
\bm{u}(x)=0, \ \bm{T_{\bm{\nu}}\bm{u}}(x)=0 & \text{ for }x=(x',x_3)\in \Gamma\times (-L,L),
\end{cases}
\end{align*}
where $\Gamma\subset \p D$ consists of two connected segments and $\mu>0, 3\lambda+2\mu>0$. Consider $\phi\in C^\infty_c(-L,L)$ and $\xi \in \R$, and we define the dimension reduction operator $\bm{R}_{\xi}$ by 
\begin{align*}
\bm{R}_{\xi}g(x'):=\int_{-L}^L e^{-ix_3 \xi }\phi(x_3)g(x',x_3)dx_3, \text{ for }x'\in D.
\end{align*}
Then one has $\bm{R}_{\xi}\bm{u}\in H^2(D;\C^3)\cap C^\alpha(D;\C^3)$ and there is a function $\bm{F}_{\xi}=\bm{F}_\xi (x')\in C^\alpha(D;\C^3)$ such that $\bm{R}_\xi \bm{u}$ is a solution of 
\begin{align}\label{Reduction elasticity system}
\begin{cases}
\widetilde{\bm{\mathcal{L}}}(\bm{R}_\xi\bm{u}(x'))=\bm{F}_\xi (x') &\text{ for }x'\in D,\\
\bm{R}_\xi u (x')=0, \ \bm{T_\nu}\left(\bm{R}_\xi\bm{u}\right)=0 &\text{ for }x'\in \Gamma,
\end{cases}
\end{align}
where 
\begin{align}\label{reduced elastic operator}
\widetilde{\bm{\mathcal{L}}}:=\left(\begin{matrix}
\lambda \Delta' + (\lambda+\mu)\partial_1^2 & (\lambda+\mu)\partial_1\partial_2 & 0 \\ (\lambda+\mu) \partial_1 \partial_2 &
\lambda \Delta' + (\lambda+\mu)\partial_2^2 & 0 \\
0 & 0 & \lambda\Delta'
\end{matrix}\right)
\end{align}
with $\Delta':=\p_1^2+\p_2^2$ being the Laplace operator with respect to the $x'$-variables. Furthermore, we have 
\begin{align}\label{reduction equation}
	\bm{F}_\xi (x') = \bm{R}_\xi f (x')\text{ for }x'\in \Gamma.
\end{align}
\end{lemma}

\begin{proof}
Denote $\bm{u}=(u_\ell)_{\ell=1}^3$, by using \cite[Lemma 3.4]{Bsource}, then one can conclude that $\bm{R}_\xi u_\ell\in H^2(D)\cap C^\alpha(D)$ for $\ell=1,2,3$. Hence, it remains to show that $\bm{R}_\xi \bm{u}$ solves \cref{Reduction elasticity system}.

In order to derive the equation for $\bm{R}_\xi\bm{u}$, note that in the three-dimensional case, the isotropic elastic operator $\bm{\mathcal{L}}$ can be rewritten as 
$$
\bm{\mathcal{L}}=\left(\begin{matrix}
\lambda \Delta + (\lambda+\mu)\partial_1^2 & (\lambda+\mu)\partial_1\partial_2 & (\lambda+\mu)\p_1\p_3 \\ (\lambda+\mu) \partial_1 \partial_2 &
\lambda \Delta + (\lambda+\mu)\partial_2^2 & (\lambda+\mu)\p_2\p_3 \\
(\lambda+\mu)\p_1\p_3 & (\lambda+\mu)\p_2\p_3 & \lambda \Delta +(\lambda+\mu)\p_3 ^2
\end{matrix}\right),
$$
then we also have $\bm{\widetilde{\mathcal{L}}}\bm{u}=\bm{f}-\bm{h}(\bm{u})$, where $$\bm{h}(\bm{u})=\left(\begin{matrix}
\lambda \p_3^2 u_1+(\lambda+\mu)\p_3\p_1 u_3 \\
\lambda \p_3^2 u_2+(\lambda+\mu)\p_3\p_2 u_3 \\
(2\lambda+\mu) \p_3^2 u_3+(\lambda+\mu)\p_3(\p_1 u_1+ \p_2 u_2)
\end{matrix}\right).
$$ 
The Lebesuge dominated convergence theorem and an integration by parts formula yield that 
\begin{align*}
\widetilde{\bm{\mathcal{L}}}(\bm{R}_\xi\bm{u})= \bm{F}_\xi(x'):=&\bm{R}_\xi\bm{f}(x')+I_\xi (x') + II_\xi (x'),
\end{align*}
where 
\begin{align*}
I_\xi (x') = &  -\int_{-L}^L e^{-ix_3 \xi} \phi''(x_3)\left(\begin{matrix}
\lambda u_1 \\ \lambda u_2 \\ (2\lambda + \mu)u_3 
\end{matrix}\right)(x',x_3)dx_3 \\
&+2i\xi \int_{-L} ^L e^{-ix_3 \xi}\phi'(x_3)\left(\begin{matrix}
\lambda u_1 \\ \lambda u_2 \\ (2\lambda + \mu)u_3 
\end{matrix}\right)(x',x_3)dx_3 \\
&+\xi^2\bm{R}_\xi\left(\begin{matrix}
\lambda u_1 \\ \lambda u_2 \\ (2\lambda + \mu)u_3 
\end{matrix}\right)(x'), 
\end{align*}
and 
\begin{align*}
II_\xi (x')= &- i\xi(\lambda+\mu) \bm{R}_\xi \left(\begin{matrix}
\p_1 u_3 \\ \p_2 u_3 \\ \p_1u_1 + \p_2 u_2
\end{matrix}\right)(x')\\
&+(\lambda+\mu)\int_{-L}^L e^{-ix_3\xi} \phi'(x_3)\left(\begin{matrix}
\p_1 u_3 \\ \p_2 u_3 \\ \p_1u_1 + \p_2 u_2
\end{matrix}\right)(x',x_3)dx_3.
\end{align*}
It is easy to see that $I_\xi (x')=0$ for $x'\in \Gamma$ since $u(x',x_3)=0$ for $(x ',x_3)\in \Gamma \times (-L,L)$. 

Next, we want to show that $ II_\xi (x')=0$ for $x'\in \Gamma$. By denoting $\Gamma := S_1 \cup S_2$, where $S_1$, $S_2$ are segments and $\overline{S_1}\cap \overline{S_2}=\{x_0'\}$ is the corner point, we only need to demonstrate that $II_\xi (x')=0$ on $S_1$. By choosing suitable boundary normal coordinates, without loss of generality, we may assume that $S_1 \times (-L,L) \subset \text{span}\{e_1,e_2\}\subset \R^3$ with its normal direction $\nu = e_3$. Here $\{e_1,e_2,e_3\}$ forms the standard orthonormal basis in $\R^3$. Recall that $u \in H^2 _{loc}(\R^3)$, then one has $\frac{\p u_j}{\p x_k}\in H^1_{loc}(\R^3)$ for $j,k\in \{1,2,3\}$. Therefore, $\frac{\p u_j}{\p x_k}|_{\Gamma\times (-L,L)}$ is a well-defined $L^2(\Gamma\times (-L,L))$-function in the trace sense. 
Since $u=0$ on $S_1 \times (-L,L)$, we have $\frac{\p u_j}{\p x_k}=0$ for $j=1,2,3$ and $k=1,2$. Therefore, by using the boundary traction $\bm{T}_{\bm{\nu}}\bm{u}=0$ on $S_1\times (-L,L)$, and that $\mu>0, \lambda+2\mu > 0$ which follow from the assumptions, one can easily see that $\frac{\p u_j}{\p x_k}=0 $ on $\Gamma \times (-L,L)$ for $j,k=1,2,3$. Similar arguments hold when $x'\in S_2$, which proves that $II_\xi(x')=0$ on $\Gamma$. This demonstrates \cref{reduction equation}.
\end{proof}

\begin{remark}
The boundary normal coordinates are useful in studying the inverse boundary value problem. The idea is based on the invariance of the elasticity system via change of variables. For example, in \cite{lin2017boundary,de2017reconstruction}, the authors utilized this technique to study the boundary determination for the isotropic elasticity system from the boundary measurements.
\end{remark}

\begin{proposition} \label{3Dprop}
Let $\Omega \subset \R^3$ be a bounded domain with $0\in \p \Omega$. Let $\theta_m$, $\theta _M$ be the number given by \cref{2Dprop} and $\mathcal{K}$ be the cone defined by \cref{cone}. Suppose that $\Omega$ has an edge of opening angle $\theta_M-\theta_m$, that is, given an origin-centred ball $B\subset \R^2$ and there exists $L>0$ such that 
\begin{align*}
(B\times (-L,L))\cap \Omega = (B\cap \mathcal{K}) \times (-L,L).
\end{align*}

Given $\bm{f}\in C^\alpha ((B\times (-L,L))\cap \Omega;\C^3)$ for some $\alpha \in (0,1)$, let $\bm{u}\in H^2((B\times (-L,L))\cap \Omega;\C^3)$ be a solution of $\bm{\mathcal{L}}\bm{u}=\bm{f}$ in $B\cap \Omega$. Then 
\begin{align*}
\bm{u}=\bm{T}_{\bm{\nu}}\bm{u}=0 \text{ on } B\cap \p \Omega \text{ implies that }\bm{f}(0)=0.
\end{align*}
\end{proposition}
\begin{proof}
The proof is similar to the proof of \cite[Proposition 3.5]{Bsource}. For the sake of completeness, we offer a detailed proof here. By the Sobolev embedding, we know that $H^2$ embeds to $C^{1/2}$ and we may assume $\alpha \leq 1/2$ without loss of generality.
By \cref{Lem dim reduce}, given any $\xi \in \R$, there is a $\bm{F}_\xi \in C^\alpha(\overline{B\cap \mathcal{K}})$, such that one can find a solution $\bm{U}\in H^2(B\cap \mathcal{K})\cap C^\alpha (\overline{B\cap \mathcal{K}})$ fulfilling $\widetilde{\bm{\mathcal{L}}}\bm{U}=\bm{F}_\xi$ in $B\cap \mathcal{K}$, where $\widetilde{\bm{\mathcal{L}}}$ is defined by \cref{reduced elastic operator}. In addition, $\bm{U}=\bm{T}_{\bm{\nu}}\bm{U}=0$ on $B\cap \p \mathcal{K}$, via \cref{2Dprop}, then we obtain $\bm{F}_\xi (0)=0$. Finally, recall that 
$$
0=\bm{F}_\xi (0)=\int_{-L}^L e^{-ix_3 \xi }\phi (x_3)\bm{f}(0,x_3)dx_3,
$$
for any smooth cut-off functions $\phi(x_3)\in C^\infty_c((-L,L))$ and for any $\xi \in \R$. The Fourier inversion formula implies that $\bm{f}(0)=0$.

\end{proof}

\section{Proof of Theorems}\label{Section 4}
In the end of this paper, we prove our theorems which stated in \cref{Section 1}.

\begin{proof}[Proof of \cref{thm1}]
  Rellich's lemma for the Helmholtz equation (see e.g. Lemma~2.11 in
  \cite{CK}) and the unique continuation principle imply that
  $\bm{u}_p = \bm{u}_s = 0$ in the connected component of
  $\R^n\setminus\overline{\Omega}$ that reaches infinity. Hence
  $\bm{u}=0$ and $\bm{T_\nu u}=0$ on the boundary of the corner or
  edge. The claim follows from \cref{2Dprop} or \cref{3Dprop}.
\end{proof}

\begin{proof}[Proof of \cref{thm2}]
  By Rellich's lemma for the Helmholtz equation again and the unique continuation principle $\bm{u}_p=\bm{u'}_p$,
  $\bm{u}_s = \bm{u'}_s$ in $\R^n\setminus
  \overline{\Omega\cup\Omega'}$. Assume
  $\Omega\not\subset\Omega'$. Then by convexity there is a corner (2D)
  or edge (3D) point
  $x_c\in\partial\Omega\setminus\overline{\Omega'}$. Since
  $\bm{u}=\bm{u'}$ outside $\overline{\Omega\cup\Omega'}$ we have
  $\bm{u}=\bm{u'}$ and $\bm{T_\nu u} = \bm{T_\nu u'}$ on
  $\partial\Omega$ near $x_c$. Set $\bm{w} = \bm{u}-\bm{u'}$. We have
  \[
  \lambda\Delta \bm{w} + (\lambda+\mu)\nabla \nabla\cdot\bm{w} +
  \omega^2\bm{w} = \bm{f}
  \]
  in $\Omega$ near $x_c$ with $\bm{w}\in H^2$. \Cref{2Dprop} and
  \cref{3Dprop} imply that $\bm{\varphi}(x_c)=0$. But this is a
  contradiction since $\bm{\varphi}\neq0$ on $\partial\Omega$. Hence
  $\Omega\subset\Omega'$. The same proof with $\Omega,\Omega'$
  switched shows that $\Omega'\subset\Omega$. Hence $\Omega=\Omega'$.

  \smallskip
  Next, let $x_c$ be a vertex (2D) or an edge point (3D) of
  $\partial\Omega = \partial\Omega'$. If $\bm{w}=\bm{u}-\bm{u'}$ then
  this time
  \[
  \lambda\Delta \bm{w} + (\lambda+\mu)\nabla \nabla\cdot\bm{w} +
  \omega^2\bm{w} = \bm{f}-\bm{f'}
  \]
  in $\Omega$ with $\bm{w}\in H^2$. Rellich's lemma for the Helmholtz
  equation and the unique continuation
  principle for the Navier equations imply that $\bm{w}=0$ and
  $\bm{T_\nu w} = 0$ on $\partial\Omega$ near $x_c$ in this case
  too. \cref{2Dprop} and \cref{3Dprop} imply $\bm{f} = \bm{f'}$ at
  $x_c$.
\end{proof}

Finally, we can prove the third main theorem in this paper.

\begin{proof}[Proof of \cref{ITP}]
  Set $\bm{f}=-\omega^2 V\bm{v}$ and $\bm{u}=\bm{v}-\bm{w}$. In two
  and three dimensions $H^2$ embeds into $C^\alpha$ for $0<\alpha<1/2$
  the latter of which we may assume. So $\bm{v}-\bm{w}$ is
  H\"older-continuous, and thus both $\bm{v}$ and $\bm{w}$ are too if
  either one is in $C^\alpha$. These functions satisfy
  \[
  \lambda \Delta \bm{u} + (\lambda+\mu)\nabla \nabla\cdot
  \bm{u}+\omega^2 \bm{u} = \bm{f}
  \]
  with $\bm{u}\in H^2(\Omega)$, $\bm{u}=\bm{T_\nu u}=0$ on
  $\partial\Omega$, and $\bm{f} \in
  C^\alpha(\overline\Omega)$. \Cref{2Dprop} and \cref{3Dprop} imply
  that $\bm{f}(x_c) = 0$, so if $V(x_c)\neq0$ then $\bm{v}(x_c)=0$ and
  since $\bm{v}=\bm{w}$ on $\partial\Omega$, so is $\bm{w}(x_c)=0$.
\end{proof}

\section{Appendix} \label{appendix}

In the end of this paper, we demonstrate that a non-zero source function $\bm{f}$ may generate zero elastic far fields $\bm{u}_s^\infty$ and $\bm{u}_p^\infty$ in $\R^3$. Recall that the elastic far fields are given by
$$
\bm{u}_s^\infty (\bm{e})= \Pi_{\bm{e}^\perp}\left(\int_{\R^3}e^{-i\omega_s\bm{e}\cdot y}\bm{f}(y)dy\right), \ 
\bm{u}_p^\infty (\bm{e})= \Pi_{\bm{e}}\left(\int_{\R^3}e^{-i\omega_p\bm{e}\cdot y}\bm{f}(y)dy\right),
$$
for direction $\bm{e}\in \mathbb{S}^2$, where $\Pi_{\bm{e}}$ is the projection with respect to $\bm{e}$. Now, simply take $\bm{f}:=e_1 \chi _{B(0,1)}$, where
\[
\chi_{B(0,1)}=\begin{cases}
1, &\text{ in }B(0,1)\\
0, &\text{ otherwise }
\end{cases}
\]
is the characteristic function and $e_1 =(1,0,0)$.

After taking this special source function $\bm{f}=e_1 \chi _{B(0,1)}  \in L^\infty(\R^3;\C^3)$, we want to find some suitable Lam\'e parameters $(\lambda, \mu )$ such that $\bm{u}_s^\infty(\bm{e})=0$ and $\bm{u}_p^\infty(\bm{e})=0$ for any direction $\bm{e} \in \mathbb{S}^{2}$. 
Moreover, the integral
$$
\int_{\R^3} e^{-i\omega_s \bm{e}\cdot y}\chi _{B(0,1)}dy = \int_{\R^3} e^{-2\pi i\left(\frac{\omega_s}{2\pi}\bm{e}\right)\cdot y}\chi_{B(0,1)}dy
$$
can be regarded as the Fourier transform of the characteristic function $\chi_{B(0,1)}$.
By using the representation formula in \cite[Appendix B.5]{grafakos2008classical}, we have 
\begin{align}\label{equation omega_s}
	\int_{\R^3} e^{-i\omega_s \bm{e}\cdot y}\chi _{B(0,1)}dy=\dfrac{J_{\frac{3}{2}}(\omega_s )}{\left(\frac{\omega_s}{2\pi}\right)^{3/2}},
\end{align}
where $J_{\frac{3}{2}}$ is the \emph{Bessel function} of order $\frac{3}{2}$.
Similarly, one also has 
\begin{align}\label{equation omega_p}
\int_{\R^3} e^{-i\omega_p \bm{e}\cdot y}\chi _{B(0,1)}dy=\dfrac{J_{\frac{3}{2}}(\omega_p)}{\left(\frac{\omega_p}{2\pi}\right)^{3/2}}.
\end{align}

Therefore, from \cref{equation omega_s,equation omega_p}, having the elastic far fields $\bm{u}_s^\infty = \bm{u}_p^\infty =0$ can be reduced to seeking the Lam\'e parameters $(\lambda,\mu)$ such that 
\begin{align}\label{appendix equ 1}
J_{\frac{3}{2}}(\omega_s)=J_{\frac{3}{2}}(\omega_p )=0,
\end{align}
where $ \omega_{s}={\omega}/{\sqrt{\mu}}$ and $\omega_{p}={\omega}/{\sqrt{\lambda+2\mu}}$.
Let $A,B>0$ be any solutions of $J_{\frac{3}{2}}(t)=0$ with $B$ greater than $A$ large enough. Combining with \cref{appendix equ 1}, without loss of generality, then we can choose 
$$
A=\dfrac{\omega}{\sqrt{\lambda+2\mu}} \text{ and }B=\dfrac{\omega}{\sqrt{\mu}}.
$$
This is equivalent to taking the Lam\'e parameters as
$$
\lambda=\left(\frac{\omega}{A}\right)^2-2\left(\frac{\omega}{B}\right)^2 \text{ and }
\mu = \left(\frac{\omega}{B}\right)^2.
$$
Note that by choosing $B$ large enough, the Lam\'e parameters $\lambda$ and $\mu$ satisfy the strong convexity condition \cref{strong convexity condition}.
Hence, for any frequency $\omega>0$, there are examples $(\lambda,\mu)$ and source $\bm{f}$ that is non-zero, but its far fields are zero. 

\textbf{Acknowledgment.}  The authors appreciate Professor Ikehata and the anonymous referees for some useful comments to improve this work.  Y.-H. Lin is partially supported by the Academy of Finland, under the project number 309963.

\bibliographystyle{abbrv}
\bibliography{ref}{}

\end{document}